\documentclass[12pt]{amsart}
\usepackage{amsmath,amssymb,amsfonts,amsthm,amsopn,dsfont}
\usepackage[dvips]{graphicx}
\usepackage{srcltx}
\usepackage{setspace}
\newtheorem{theorem}{Theorem}[section]

\newtheorem{definition}[theorem]{Definition}
\newtheorem{example}[theorem]{Example}
\newtheorem{lemma}[theorem]{Lemma}
\numberwithin{equation}{section}

\newtheorem{proposition}[theorem]{Proposition}
\newtheorem{remark}[theorem]{Remark}

\renewcommand{\ell}{l}
\renewcommand{\epsilon}{\varepsilon}

\def\C{\mathbb{C}}
 \def\cS{\mathcal{S}}

\def\rd{\bR^d}

\def\R{\right)}

\def\<{\left<}
\def\>{\right>}

\def\mv1{M_v^1}

\def\mn{(m,n)}
\def\mn'{(m',n')}

\def\N{\mathbb{N}}
\def\R{\mathbb{R}}
\def\C{\mathbb{C}}
\def\rd{\mathbb{R}^d}

\begin{document}

\begin{abstract} For a class of ordinary differential
operators $P$ with polynomial coefficients, we give a
necessary and sufficient condition for $P$ to be globally
regular in $\R$, i.e. $u\in\cS^\prime(\R)$ and
$Pu\in\cS(\R)$ imply $u\in \cS(\R)$ (this can be regarded as a global
version of the Schwartz' hypoellipticity notion). The
condition involves the asymptotic behaviour, at infinity,
of the roots $\xi=\xi_j(x)$ of the equation $p(x,\xi)=0$,
where $p(x,\xi)$ is the (Weyl) symbol of $P$.
\end{abstract}

\title[]{Global regularity for ordinary differential operators with polynomial coefficients}
\author{Fabio Nicola \and Luigi Rodino}
\address{Dipartimento di Matematica, Politecnico di
Torino, Corso Duca degli
Abruzzi 24, 10129 Torino,
Italy}
\address{Dipartimento di Matematica,  Universit\`{a} degli Studi di Torino,
Via Carlo Alberto 10, 10123
Torino, Italy}

\email{fabio.nicola@polito.it}
\email{luigi.rodino@unito.it} 
\date{}

\maketitle

\section{Introduction and statement of the result}
Linear ordinary differential operators with polynomial
coefficients play an important role in mathematics. On one
hand, they were source for the study of classes of special
functions, with links all around to Applied Sciences. On
the other hand, the general study of Fuchs points,
irregular singular points and Stokes phenomena present deep
connections with different branches of geometry.\par We may
write the generic operator in the form
\begin{equation}\label{1.1}
P=\sum_{\alpha+\beta\leq m} a_{\alpha,\beta} x^\beta D^\alpha, \qquad a_{\alpha,\beta}\in\C,
\end{equation}
where we use the notation $D=-id/dx$. In this paper we deal
with the Fourier analysis of \eqref{1.1}. Namely, rather
than analysing the extension of the solution $u$ of the
corresponding equation in the complex domain $\C$, we shall
address to their analysis in $\R^2_{x,\xi}$, with respect
to the phase-space variables $x,\xi\in\R$. Basic function
space in this order of ideas is the space $\cS(\R)$ of L.
Schwartz \cite{schwartz}, defined by imposing
\begin{equation}
\sup_{x\in\R} |x^\beta D^\alpha u(x)|<\infty,\qquad \forall \alpha,\beta\in\mathbb{N}.
\end{equation}
These functions are regular in $\R^2_{x,\xi}$, in the sense that both $u(x)$ and its Fourier transform $\widehat{u}(\xi)=\int_\R e^{-ix\xi} u(x)\,dx$ present a rapid decay at infinity, beside local regularity. As universal set for our study we shall take the dual $\cS^\prime(\R)$; note that this will exclude solutions which have exponential growth at infinity.\par
Aim of this paper is to characterize the operators $P$ in \eqref{1.1} which are globally regular, according the following definition.
\begin{definition}
We say that $P$ is globally regular if
\begin{equation}\label{1.3}
u\in\cS^\prime(\R)\ {\it and}\ Pu\in\cS(\R)\Rightarrow u\in \cS(\R).
\end{equation}
\end{definition}
In particular, if \eqref{1.3} is satisfied, the solutions
$u\in\cS^\prime(\R)$ of $Pu=0$ belong to $\cS(\R)$. Global
regularity turns out to be basic information in many
applications. So for example in connection with Quantum
Mechanics, assuming $P$ in \eqref{1.1} is self-adjoint, we
may deduce that the eigenfunctions, intended as solutions
$u\in L^2(\R)$ of $Pu=0$, are in $\cS(\R)$. In the Theory
of Signals, where we may regard $P$ in \eqref{1.1} as a
filter reproduced by electronic devices, we have that the
globally regular $P$ are exactly the ineffective filters,
i.e. filters which do not cancel any essential part of the
signal.\par The literature concerning global regularity,
sometimes also called global hypoellipticity, in these last
30 years is extremely large, taking also into account the
same problem for operators with smooth coefficients and
pseudo-differential operators in $\R^n$, with $n\geq1$. We
address to the recent monograph of the authors
\cite{nicola} for a survey.\par Let us begin with a simple
example. The constant coefficient operator
$p(D)=\sum_{\alpha\leq m} c_\alpha D^\alpha$,
$c_\alpha\in\C$, is globally regular if and only if
$p(\xi)\not=0$ for all $\xi\in\R$. In fact, if
$p(\xi_0)=0$, then $u(x)= {\rm exp}[i\xi_0 x]$ provides a
solution of $p(D) u=0$, whereas by Fourier transform one
gets easily that $p(D)u\in\cS(\R)$, $u\in\cS^\prime(\R)$
imply $u\in\cS(\R)$ if $p(\xi)\not=0$ for all $\xi\in\R$.
The same result keeps valid for partial differential
operators with constant coefficients.\par Passing to
operators with polynomial coefficients, a characterization
of globally regular operators in $\R^n$ is certainly out of
reach at this moment. However for the ordinary differential
operator \eqref{1.1} a necessary and sufficient condition
seems possible, and we shall give it in the following,
under an additional algebraic condition.\par Consider first
the standard left-symbol of $P$ in \eqref{1.1}:
\begin{equation}\label{1.4}
a(x,\xi)=\sum_{\alpha+\beta\leq m} a_{\alpha,\beta} x^\beta \xi^\alpha.
\end{equation}
In our approach, it will be convenient to argue on the Weyl
symbol, see for example \cite[Chapter XVIII]{hormanderIII},
given by
\begin{align}\label{1.5}
p(x,\xi)&=\sum_{\alpha+\beta\leq m} c_{\alpha,\beta} x^\beta \xi^\alpha\\
&=\sum_{\gamma\geq0} \frac{1}{\gamma!}\left(-\frac{1}{2}\right)^\gamma \partial^\gamma_\xi D^\gamma_x a(x,\xi).\nonumber
\end{align}
We shall assume $c_{m,0}=1$ in \eqref{1.5}, i.e. $a_{m,0}=1$ in \eqref{1.4}.
 We have
\begin{equation}\label{fact}
p(x,\xi)=\prod_{j=1}^m (\xi-\xi_j(x))
\end{equation}
where $\xi_j(x)$, $j=1,\ldots,m,$ are real-analytic functions defined for $x\in\R$, $|x|$ large enough. Let us denote by $\lambda_1,\ldots,\lambda_m$ the complex roots of the polynomial
\[
\sum_{\alpha=0}^m c_{\alpha,m-\alpha} \lambda^{\alpha}.
\]
We have, possibly after relabeling,
\begin{equation}\label{fact1}
\xi_j(x)/x\to \lambda_j\quad {\rm as}\ x\in\R,\ |x|\to+\infty.
\end{equation}
In fact, the roots have a Puiseaux expansion at infinity (\cite[Lemma A.1.3, page 363]{hormanderII}), namely
\begin{equation}\label{puiseaux}
\xi_j(x)=\lambda_j x+\sum_{- \infty<k\leq p-1}
c_{j,k}\big(x^{1/p}\big)^k\qquad {\rm for}\ |x|\ {\rm
large},
\end{equation} for some integer $p$, where the
function $x^{1/p}$ is the positive $p$th root of $x$ for
$x>0$ and, say, $x^{1/p}=|x|^{1/p} e^{i\pi/p}$ for $x<0$
(by taking the the lowest common multiple we can assume
that the same integer $p$ occurs for every $j$).
 \par

We suppose that the roots which approach the real axis at infinity are asymptotically separated, in the sense that
\begin{equation}\label{separazione}
{\it If}\ \lambda_j=\lambda_k\in\R,\ {\it with}\ j\not=k,\ {then}
\end{equation}
\[
 |\xi_j(x)-\xi_k(x)|\gtrsim\max\{|\xi_j(x)-\lambda_j x|,|\xi_k(x)-\lambda_k x|,|x|^{-1+\epsilon}\}
\]
for some $\epsilon>0$.
\begin{theorem}\label{mainteo}
Assume \eqref{separazione}. Then $P$ is globally regular, i.e. \eqref{1.3} holds, if and only if
\begin{equation}\label{condition}
|x\,{\rm Im}\,\xi_j(x)|\to+\infty\qquad{\rm as}\ \ x\in\R,\ |x|\to+\infty.
\end{equation}
\end{theorem}
For a better understanding of \eqref{separazione}, \eqref{condition}, we may argue on the Puiseaux expansions \eqref{puiseaux}. Let us emphasize the first non-vanishing term after $\lambda_j x$, namely write
\begin{equation}\label{1.x}
\xi_j(x)=\lambda_j x+c_{j,r(j)} \big(x^{1/p}\big)^{r(j)}+\sum_{-\infty<k<r(j)} c_{j,k} \big(x^{1/p}\big)^k
\end{equation}
with $c_{j,r(j)}\not=0$. Arguing for simplicity for $x>0$, condition \eqref{separazione} states that, if $\lambda_j=\lambda_k\in\R$, then one at least between $r(j)$ and $r(k)$ is strictly larger than $-1$ and moreover, in the case $r(j)=r(k)>-1$, we have $c_{j,r(j)}\not=c_{k,r(k)}$.\par
As for \eqref{condition}, arguing again for $x>0$, it means that for all $j=1,\ldots,m$, we have ${\rm Im}\, \lambda_j\not=0$ or else ${\rm Im}\, c_{j,k}\not=0$ for some $k$ with $k/p>-1$. To be definite: \eqref{condition} is not satisfied when for some $j$ the function
\[
\lambda_j x+\sum_{-p<k<p-1} c_{j,k} \big(x^{1/p}\big)^k
\]
is real-valued.\par  In the next Section 2,
we recall some notation for the pseudodifferential calculus
and we reduce the problem to the analysis of a global wave
front set. The proof of Theorem \ref{mainteo} is given in
Section 4, by using the factorization of the operator
presented in Section 3. In some sense, this factorization
is the counterpart in phase-space of the classical methods
of the asymptotic integration. In fact, we think that, in terms of classical asymptotic analysis, under condition \eqref{separazione}, a proof of Theorem \ref{mainteo} would be as well possible, but much more difficult. Section 5 is devoted to
remarks and examples. In particular, we recapture some
existing results on the global regularity for operators of
the form \eqref{1.1}, giving corresponding references.\par
We finally observe that Theorem
\ref{mainteo} keeps valid for relevant classes of operators
of the form \eqref{1.1}, independently of the asymptotic
separation \eqref{separazione}. For example,
\eqref{separazione} is not satisfied for constant
coefficients operators with a double root
$\xi_j(x)=\xi_k(x)=\xi_0$, whereas in this class
\eqref{condition} is necessary and sufficient for the
global regularity, as observed before. Regretfully, the
assumption \eqref{separazione} will be essential in our
proof of Theorem \ref{mainteo} for operators with
polynomial coefficients.

\section{Microlocal reduction}
We recall that a pseudodifferential operator in $\R$,
according to the standard quantization, is an integral
operator of the form
\[
p(x,D) u(x)= (2\pi)^{-1}\int_{\R} e^{ix\xi} p(x,\xi)
\widehat{u}(\xi)\,d\xi,
\]
for $u\in\cS(\R)$, where the so-called symbol $p(x,\xi)$ is
a smooth function in $\R^2$ which satisfies suitable growth
estimates at infinity, which will be detailed below. The corresponding
operator $p(x,D)$ will define continuous maps
$\cS(\R)\to \cS(\R)$ and $\cS'(\R)\to\cS'(\R)$.\par
An important class of symbols is given by the space
$\Gamma^m(\R)$ of smooth functions $p(x,\xi)$ in $\R^2$
satisfying the estimates
\[
|\partial^\alpha_\xi\partial^\beta_x p(x,\xi)|\leq
C_{\alpha\beta}(1+|x|+|\xi|)^{m-\alpha-\beta},\qquad\forall\alpha,\beta\in\mathbb{N},\ (x,\xi)\in\R^2,
\]
for some $m\in\R$. This class arises, in particular, in the following definition of the global wave-front set of a temperate distribution, as introduced in \cite{hormanderhyper}.\par A point
$(x_0,\xi_0)\in\R^2\setminus\{0\}$ is called {\it non-characteristic}
for $p\in\Gamma^m(\R)$, if there are $\epsilon,C>0$ such
that
\begin{equation}\label{med11}
|p(x,\xi)|\geq C(1+|x|+|\xi|)^m\quad {\rm for}\ (x,\xi)\in
V_{(x_0,\xi_0),\epsilon}
\end{equation}
 where, for $z_0\in\R^2$, $z_0\not=0$, $V_{z_0,\epsilon}$ is the conic neighborhood
\[
V_{z_0,\epsilon}=\Big\{z\in\R^2\setminus\{0\}:\
\Big|\frac{z}{|z|}-\frac{z_0}{|z_0|}\Big|<\epsilon,\
|z|>\epsilon^{-1}\Big\}.
\]
Let $u\in\cS'(\R)$. We define its (global) wave-front set
$WF(u)\subset\R^{2}\setminus\{0\}$ by saying that
$(x_0,\xi_0)\in\R^{2}$, $(x_0,\xi_0)\not=(0,0)$, does not belong to $WF(u)$ if there
exists $\psi\in\Gamma^0(\R)$ which is non-characteristic at
$(x_0,\xi_0)$, such that $\psi(x,D) u\in\cS(\R)$. The set $WF(u)$ is a conic closed subset of $\R^2\setminus\{0\}$.
\begin{proposition}{\rm (\cite{hormanderhyper})}\label{pro16}
If the point $(x_0,\xi_0)$ is non-characteristic for $p$, and $(x_0,\xi_0)\not\in WF(p(x,D)u)$, then $(x_0,\xi_0)\not\in WF(u)$.
\end{proposition}
\begin{proposition}{\rm (\cite{hormanderhyper})}\label{pro2.1bisbis}
For every $p\in \Gamma^m(\R)$, and $u\in\cS^\prime(\R)$, we have $WF(p(x,D)u)\\ \subset WF(u)$.
\end{proposition}

We will also use the following characterization of the Schwartz class.
\begin{proposition}{\rm (\cite{hormanderhyper})}\label{pro16bis}
For $u\in\cS^\prime(\R)$ we have $WF(u)=\emptyset$ if and only if $u\in\cS(\R)$.
\end{proposition}

In fact, results similar to those in Proposition \ref{pro16} hold for more general classes of operators. The following case will be important in the following. \par
Consider the class $\tilde{\Gamma}^m_\delta(\R)$ of symbols $p(x,\xi)$ satisfying the following estimates:
\begin{equation}\label{med12}
|\partial^\alpha_\xi\partial^\beta_x p(x,\xi)|\leq
C_{\alpha\beta}\langle \xi\rangle^{m-\alpha}\langle x\rangle^{m-\beta+\delta\alpha},\qquad\forall\alpha,\beta\in\mathbb{N},\ (x,\xi)\in \R^2,
\end{equation}
for some $m\in\R$, $0\leq \delta<1$. For $t\in\R$, we write $\langle t\rangle=(1+t^2)^{1/2}$. Notice that $\Gamma^0(\R)\subset\tilde{\Gamma}^0(\R)$. These classes are a special case of the Weyl-H{\"o}rmander classes $S(M,g)$, with weight $M(x,\xi)=\langle x\rangle^m\langle \xi\rangle^m$ and metric $g_{x,\xi}=\frac{dx^2}{\langle x\rangle^2}+\frac{d\xi^2}{\langle x\rangle ^{-2\delta}\langle \xi\rangle^2}$. The usual symbolic calculus with full asymptotic expansions works for these classes (cf.\ \cite[Chapter XVIII]{hormanderIII}), because the so-called Planck function $h(x,\xi)=\langle \xi\rangle ^{-1}\langle x\rangle^{\delta-1}$ satisfies $h(x,\xi)\leq C(1+|x|+|\xi|)^{\delta-1}$ and $\delta-1<0$.  \par
A symbol $p\in \tilde{\Gamma}^m_\delta(\R)$ is called hypoelliptic at $(x_{0},\xi_{0})\not=(0,0)$  if for some $\epsilon>0$ it satisfies
\begin{equation}\label{med13}
|\partial^\alpha_\xi\partial^\beta_x p(x,\xi)|\leq
C_{\alpha\beta}|p(x,\xi)|\langle \xi\rangle^{-\alpha}\langle x\rangle^{-\beta+\delta\alpha},\quad\forall\alpha,\beta\in\mathbb{N},\quad (x,\xi)\in V_{(x_0,\xi_0),\epsilon}.
\end{equation}
and
\begin{equation}\label{med13g}
|p(x,\xi)|\geq C\langle x\rangle^{m^\prime}\langle \xi\rangle ^{m^\prime}\quad\forall  (x,\xi)\in V_{(x_0,\xi_0),\epsilon},
\end{equation}
for some $m^\prime\in\R$, $C>0$.
\begin{proposition}\label{pro17}  Let $p(x,\xi)$ satisfy \eqref{med12}, \eqref{med13} and \eqref{med13g}.
If $(x_0,\xi_0)\not\in WF(p(x,D)u)$ then $(x_0,\xi_0)\not\in WF(u)$.
\end{proposition}
The proof is standard, since the assumptions imply the
existence of a microlocal parametrix $Q=q(x,D)$,
$q\in\tilde{\Gamma}^{-m}_\delta(\R)$ (see e.g. \cite[Lemma
3.1]{Hormander:1} and \cite[Lemma 1.1.4 and Proposition
1.1.6]{nicola}).\par

Finally, we will also make use of the the class $\Gamma^m_{1,\delta}(\R)$ of smooth functions $p\in C^\infty(\R^2)$ such that
\begin{equation}\label{pro17bis}
|\partial^\alpha_\xi \partial^\beta_x p(x,\xi)|\leq C_{\alpha\beta} (1+|x|+|\xi|)^{m-\alpha+\delta\beta},\qquad\forall\alpha,\beta\in\mathbb{N},\ (x,\xi)\in\R^2,
\end{equation}
for some $m\in\R$, $0\leq \delta< 1$. We have the following result.
\begin{proposition}\label{pro18}
Let $p\in \Gamma^{m}_{1,\delta}(\R)$. Then $WF(p(x,D) u)\subset WF(u)$.
\end{proposition}
The proof is again standard and omitted for the sake of brevity.\par
Finally we define, for future references, the following class of functions in $\R$. We set $S^m(\R)$, $m\in\R$, for the the space of smooth functions satisfying the estimates
\begin{equation}\label{do4}
|f^{(\alpha)}(x)|\leq C_{\alpha}\langle x\rangle^{m-\alpha},\qquad \forall \alpha\in\N,\ x\in\R.
\end{equation}

As already observed in the Introduction,  it is also useful to
deal with the {\it Weyl quantization}, defined as
\[
p^w(x,D)= (2\pi)^{-1}\int_{\R} e^{i(x-y)\xi}
p\Big(\frac{x+y}{2},\xi\Big) u(y)\,dy\,d\xi.
\]
For example, an operator with a symbol $a(x,\xi)$ as in \eqref{1.4} in the standard quantization can be always re-written as an operator with Weyl symbol $p(x,\xi)$ given by the formula \eqref{1.5}.\\
The
main feature of Weyl quantization is its symplectic
invariance: if $\chi:\R^2\to\R^2$ is a linear symplectic
map, there is a unitary operator $U$ in $L^2(\R)$ which is
also an isomorphism of $\cS(\R)$ into itself and of
$\cS'(\R)$ into itself (in fact a metaplectic operator),
such that
\begin{equation}\label{med10}
(p\circ \chi )^w(x,D)=U^{-1} p^w(x,D) U.
\end{equation}
Moreover, in the definition of the global wave-front set it would be
equivalent to use the Weyl quantization, which implies at once that
\begin{equation}\label{med10bis}
(x_0,\xi_0)\not\in WF(u)\Longleftrightarrow \chi(x_0,\xi_0)\not\in WF (U u),
\end{equation}
where $U$ is the operator associated to $\chi$ as in \eqref{med10}. \par

\section{Factorization of the operator}
In this section we provide a convenient microlocal factorization of the operator $P$ in Theorem \ref{mainteo}, which reduces the proof of that result to the case of first order operators.\par
Let $j\in\N$ and consider the elementary symmetric functions defined by
\begin{align*}
\sigma_0(\xi_1,\ldots,\xi_j)&=1\\
\sigma_1(\xi_1,\ldots,\xi_j)&=-\sum_{1\leq j\leq j} \xi_j\\
\sigma_2(\xi_1,\ldots,\xi_j)&=\sum_{1\leq j<k\leq j} \xi_j\xi_k\\
\ldots&\\
\sigma_j(\xi_1,\ldots,\xi_j)&=(-1)^j\xi_1\cdots\xi_j.
\end{align*}
\begin{proposition}\label{pro1}
Let $r_1\geq0$, $r_2\geq0$ and $a_j(x)$, $j=0,\ldots, r_1$, and $\xi_j(x)$, $j=r_1,\ldots,r_1+r_2$, be smooth functions, defined  in some open subset of $\R$. Then we have
\begin{multline*}
\sum_{k=0}^{r_1}  a_k(x)D^{r_1-k}\prod_{j=r_1}^{r_1+r_2}(D-\xi_{j}(x))\\ =\sum_{k=0}^{r_1+r_2} (\sum_{l+h=k\atop l\leq r_1, h\leq r_2} a_l(x)\sigma_h(\xi_{r_{1}+1}(x),\ldots,\xi_{r_{1}+r_{2}}(x))+R_k(x))D^{r_1+r_2-k}
\end{multline*}
where $R_0=R_1=0$ and, for $k\geq 2$,
\begin{multline}
R_k\in{\rm span}_{\C}\{a_l \xi_{j_1}^{(m_1)}\cdots\xi_{j_h}^{(m_h)},\ l+h+m_1+\ldots +m_h=k,\ 1\leq h\leq k-1,\\ m_1+\ldots +m_h>0,\ r_1+1\leq j_1<\ldots<j_h\leq r_1+r_2\}.
\end{multline}
\end{proposition}
\begin{proof}
Let us apply induction on $r_2$. The conclusion obviously holds with $R_j=0$ for every $j$ if $r_2=0$. Suppose it holds for $r_2$. The we have
\begin{multline}\label{pro21}
\sum_{k=0}^{r_1}  a_j(x)D^{r_1-k}(D-\xi_{r_1+1}(x))\cdots(D-\xi_{r_1+{r_2+1}}(x))=\\ \sum_{k=0}^{r_1+r_2} \Big( \sum_{l+h=k\atop l\leq r_1, h\leq r_2} a_l(x)\sigma_h(\xi_1(x),\ldots,\xi_k(x)+R_k(x)\Big)D^{r_1+r_2-k}(D-\xi_{r_1+r_2+1}(x)),
\end{multline}
with $R_k$ as in the statement. Let
$b_{j}(x)=\sum\limits_{l+h=j\atop l\leq r_1, h\leq r_2}
a_l(x)\sigma_h(\xi_{r_{1}+1}(x),\ldots,\xi_{r_{1}+r_{2}}(x))$.
By Leibniz' formula the expression in the right-hand side
of \eqref{pro21} reads \begin{multline}
\sum_{k=0}^{r_1+r_2+1} \Big( b_{k}(x)-\xi_{r_1+r_2+1}(x)b_{k-1}(x)+R_k(x)-\\
\sum_{j=1}^{k-1}\binom{r_1+r_2+1-k-j}{j}(b_{k-j-1}(x)+R_{k-j-1}(x))D^j \xi_{r_1+r_2+1}(x)\Big) D^{r_1+r_2+1-k}.
\end{multline}
Now, we have
\[
 b_k(x)-\xi_{r_1+r_2+1}(x)b_{k-1}(x)=\sum\limits_{l+h=k\atop l\leq r_1, h\leq r_2+1} a_l(x)\sigma_h(\xi_{r_{1}+1}(x),\ldots,\xi_{r_{1}+r_{2}+1}(x)),
 \] whereas the terms $R_k(x)$, $b_{k-j-1}(x)D^j \xi_{r_1+r_2+1}(x)$ and $R_{k-j-1}(x)D^j \xi_{r_1+r_2+1}(x)$ are admissible errors.
\end{proof}
\begin{proposition}\label{proweyl}
Let $r_1\geq0$, $r_2\geq0$ and $a_j(x)$, $j=0,\ldots,r_1$, and $b_h(x)$, $h=0,\ldots,r_2$ be smooth functions on $\R$. Consider the operator $P$ with Weyl symbol
\begin{equation}\label{weyl}
p(x,\xi)=\sum_{j=0}^{r_1} a_j(x) \xi^{r_1-j}\cdot \sum_{h=0}^{r_2} b_h(x) \xi^{r_2-h}.
\end{equation}
Then $P$ in the standard quantization has symbol
\[
q(x,\xi)= \sum_{k=0}^{r_1+r_2}\Big(R_k+\sum_{l+h=k\atop l\leq r_{1}, h\leq r_{2}} a_l b_h\Big)\xi^{r_1+r_2-k},
\]
where
\[
R_k\in{\rm span}_{\C}\{ (a_l b_h)^{(\nu)},\ 1\leq \nu\leq k,\ l+h=k-\nu\}.
\]
\end{proposition}

\begin{proof}
From \eqref{weyl}  we have $p(x,\xi)=\sum_{k=0}^{r_1+r_2} \sum_{l+h=k} a_l b_h\, \xi^{r_1+r_2-k}$. On the other hand,
for the standard symbol $q(x,\xi)$ we have the formula
\[
q(x,\xi)= \sum_{\nu=0}^{r_1+r_2-k} \frac{1}{\nu!}\Big(-\frac{i}{2}\Big)^\nu \partial_x^\nu\partial_\xi^\nu p(x,\xi),
\]
so that the desired result follows at once.
\end{proof}
\begin{lemma}\label{lemma1}
Let $\xi_{k}\in\mathbb{C}$, $k=1,\ldots, r_{1}+r_{2}$,
\[
Q(\xi)= \prod_{r_1+1\leq k\leq r_1+r_2} (\xi-\xi_k)
\]
and
\[
Q_j(\xi)= \prod_{1\leq k\leq r_1+r_2\atop k\not=j} (\xi-\xi_k),
\]
for  $j=r_1+1,\ldots,r_1+r_2$.\par Consider the square
matrix $A$ of size $r_1+r_2$, whose $j$th column for $1\leq
j\leq r_1$ is made of the coefficients of the polynomials
$\xi^{r_1-j}Q(\xi)$, whereas the $j$th column for
$r_1+1\leq j\leq r_1+r_2$ is made of the coefficients of
the polynomial $Q_j$ (hence in the $j$th row there are the
coefficients of $x^{r_1+r_2-j}$ which appear in those
polynomials). Assume that $\xi_k\not=\xi_j$ when $j\not=k$.
Then the matrix $A$ is invertible and its inverse
$A^{-1}=B=(B_{jk})$ is given by the following formula:
\begin{equation}\label{inversa1}
B_{jk}=\sum_{h=1}^{r_1}\frac{\xi_h^{r_1+r_2-k}}{\prod\limits_{1\leq l\leq r_1+r_2\atop l\not=h}(\xi_h-\xi_l)}\sigma_{j-1}(\xi_1,\ldots,\xi_{h-1},\xi_{h+1},\ldots,\xi_{r_1})
\end{equation}
{\rm if}\  $1\leq j\leq r_1,\ 1\leq k\leq r_1+r_2$, whereas
\begin{equation}\label{inversa2}
B_{jk}=\frac{\xi_j^{r_1+r_2-k}}{\prod\limits_{1\leq l\leq r_1+r_2\atop l\not=j}(\xi_j-\xi_l)}\end{equation}
{\rm if}\ $ r_1+1\leq j\leq r_1+r_2,\ 1\leq k\leq r_1+r_2$.
\end{lemma}
\begin{proof}
To compute the $k$th column of the inverse matrix, we have to solve the system $AX=Y$, where $Y=[0,\ldots,0,1,0\ldots,0]^t$ is the $k$th element of the canonical basis of $\C^{r_1+r_2-1}$. Now, the vector $X=[a_1,\ldots,a_{r_1+r_2}]^t$ is  a solution if and only if \begin{equation}\label{4a}
R(\xi) Q(\xi)+\sum_{j=r_1+1}^{r_1+r_2} a_j Q_j(\xi)=\xi^{r_1+r_2-k},
\end{equation}
with $R(\xi)=\sum_{j=1}^{r_1}a_j \xi^{r_1-j}$.\par
To solve \eqref{4a}, we express the right-hand side in terms of the interpolating Lagrange polynomials associated to the points $\xi_j$, $j=1,\ldots,r_1+r_2$, which are
\[
L_h(\xi):=\prod_{1\leq k\leq r_1+r_2\atop k\not=h} (\xi-\xi_k), \qquad h=1,\ldots, r_1+r_2
\]
(hence $L_h(x)=Q_h(x)$ for $h=r_1+1,\ldots,r_1+r_2$).  We have
\begin{equation}\label{interpolazione}
\xi^{r_1+r_2-k}=\sum_{h=1}^{r_1+r_2-1} \frac{\xi_h^{r_1+r_2-k}}{L_h(\xi_h)}L_h(\xi),
\end{equation}
hence we are reduced to solve \eqref{4a} when the right-hand side is replaced by each Lagrange polynomial. On the other hand, one sees immediately that the equation
\[
R(\xi) Q(\xi)+\sum_{j=r_1+1}^{r_1+r_2} a_j Q_j(\xi)=L_h(x),
\]
when $1\leq h\leq r_1$ has the solution
$R(\xi)=\prod\limits_{1\leq l\leq r_1\atop l\not=h}
(\xi-\xi_l)$, and $a_j=0$ for every
$j=r_1+1,\ldots,r_1+r_2$, whereas when $r_1+1\leq h\leq
r_1+r_2$ the solution is given by $R(x)=0$ and $a_h=1$,
$a_j=0$ if $r_1+1\leq j\leq r_1+r_2$ with $j\not=h$. \par
By linearity and \eqref{interpolazione} we obtain that the
solution of \eqref{4a} is given by
\[
\begin{cases}
R(\xi)=\sum_{j=1}^{r_1}  \frac{\xi_h^{r_1+r_2-k}}{L_h(\xi_h)}\prod\limits_{1\leq l\leq r_1\atop l\not=h} (\xi-\xi_l)\\ a_j=\frac{\xi_j^{r_1+r_2-k}}{L_j(\xi_j)},\qquad j=r_1+1,\ldots,r_1+r_2.
 \end{cases}
\]
We can now come back to the original system $AX=Y$, and we obtain the desired form for the inverse matrix.
\end{proof}

It follows from the proof of the following proposition that the result in Lemma \ref{lemma1} continues to hold if one just require that $\xi_j\not=\xi_k$ for the indices $j\not=k$ such that $j>r_1$ or $k>r_1$. In fact, as we will see, a simplification of the factors $\xi_h-\xi_l$, with $1\leq k\not= l\leq r_1$ occurs in \eqref{inversa2}.

\begin{lemma}\label{lemma2}
Let $\xi_j(x)$, $j=1,\ldots,r_1+r_2$ be smooth functions for $x$ large, such that $\xi_j(x)\not=\xi_k(x)$ if $r_1+1\leq j\not=k\leq r_1+r_2$,
\begin{equation}\label{fo1}
\xi_j(x)=O(x),\qquad {\rm if}\ 1\leq j\leq r_1+r_2
\end{equation}
and
\begin{equation}\label{fo2}
|\xi_j(x)-\xi_k(x)|\gtrsim x\qquad {\rm if}\ \ 1\leq j\leq r_1\ {\rm and}\ r_1+1\leq k\leq r_1+r_2
\end{equation}
as $x\to+\infty$. If we set $\xi_j=\xi_j(x)$ in Lemma \ref{lemma1}, for the inverse matrix $B_{jk}=B_{jk}(x)$ given there, the following asymptotic formulae hold as $x\to+\infty$:
\begin{equation}\label{fo30}
B_{jk}(x)=O(x^{j-k})
\end{equation}\label{fo3}
for $1\leq j\leq r_1$, $1\leq k\leq r_1+r_2$, whereas
\begin{equation}\label{fo4}
B_{jk}(x)=O\Bigg(\frac{|\xi_j(x)|^{r_1+r_2-k}}{x^{r_1}\prod\limits_{r_1+1\leq l\leq r_1+r_2\atop l\not=j}|\xi_j(x)-\xi_l(x)|}\Bigg)
\end{equation}
for $r_1+1\leq j\leq r_1+r_2$, $1\leq k\leq r_1+r_2$.
\end{lemma}
\begin{proof}
Consider first the entries with $1\leq j\leq r_1$. With the notation of Lemma \ref{lemma1} we can rewrite \eqref{inversa1} as
\begin{equation}\label{fo5}
B_{jk}=\frac{\sum\limits_{h=1}^{r_1}(-1)^{h-1}\xi_h^{r_1+r_2-k}\sigma_{j-1}(\xi_1,\ldots,\xi_{h-1},\xi_{h+1},\ldots,\xi_{r_1})\prod\limits_{1\leq l\leq r_1\atop l\not=h}Q(\xi_l)\prod\limits_{1\leq l<l^\prime\leq r_1\atop l,l^\prime\not=h} (\xi_l-\xi_{l^\prime})}{\prod\limits_{1\leq l<l^\prime\leq r_1} (\xi_l-\xi_{l^\prime})\prod\limits_{1\leq l\leq r_1}Q(\xi_l)}
\end{equation}
We claim that the numerator of this fraction, as a polynomial in $\xi_1,\ldots,\xi_{r_1+r_2}$ is divisible by the product $\prod\limits_{1\leq l<l^\prime\leq r_1} (\xi_l-\xi_{l^\prime})$ which arises in the denominator. To this end, it is sufficient to show that the numerator vanishes if $\xi_\mu=\xi_\nu$, for every $1\leq \mu<\nu\leq r_1$. It is clear that all the terms in the above sum with $h\not=\mu,\nu$ vanish if $\xi_\mu=\xi_\nu$, because of the presence of the factor $\xi_\mu-\xi_\nu$. Let us prove that also the sum of the two terms corresponding to $h=\mu$ and $h=\nu$ vanishes. Due to the symmetry of $\sigma_{j-1}$, it suffices to show that
\[
(-1)^{\mu-1} \prod\limits_{1\leq l<l^\prime\leq r_1\atop l,l^\prime\not=\mu} (\xi_l-\xi_{l^\prime})+(-1)^{\nu-1} \prod\limits_{1\leq l<l^\prime\leq r_1\atop l,l^\prime\not=\nu} (\xi_l-\xi_{l^\prime})=0.
\]
This amounts to prove that
\[
(-1)^{\mu-1+r_1-\nu} \prod\limits_{1\leq l\leq r_1\atop l\not=\mu,\nu} (\xi_l-\xi_\nu)+(-1)^{\nu-1+r_1-\mu-1} \prod\limits_{1\leq l\leq r_1\atop l\not=\mu,\nu} (\xi_l-\xi_\mu)=0.
\]
where we took  into account that $\mu<\nu$.  But this is true, because the two products that arise in the last formula coincide when $\xi_\mu=\xi_\nu$.  \par
This proves the claim and shows, by a limiting argument, that the the matrix $A$ in Lemma \ref{lemma1} is still invertible if one just require $\xi_j\not=\xi_k$ for the indices $j\not=k$ such that $j>r_1$ or $k>r_1$.\par
Now, by dividing the numerator of \eqref{fo5} by $\prod\limits_{1\leq l<l^\prime\leq r_1} (\xi_l-\xi_{l^\prime})$ we get a polynomial in $\xi_1,\ldots,\xi_{r_1+r_2}$ of degree\footnote{Observe that the polynomials $\prod\limits_{1\leq l<l^\prime\leq r_1\atop l,l^\prime\not=h} (\xi_l-\xi_{l^\prime})$ and $\prod\limits_{1\leq l<l^\prime\leq r_1} (\xi_l-\xi_{l^\prime})$ have degree $\frac{(r_1-1)(r_1-2)}{2}$ and $\frac{r_1(r_1-1)}{2}$ respectively.}
\[
(r_1+r_2-k)+(j-1)+(r_1-1)r_2+\frac{(r_1-1)(r_1-2)}{2}-\frac{r_1(r_1-1)}{2}=r_1r_2-k+j.
\]
Hence, when we compose such a polynomial with the functions $\xi_j=\xi_j(x)$, by \eqref{fo1} we get a function which is $O(x^{r_1r_2-k+j})$ as $x\to+\infty$. On the other hand, using \eqref{fo2} we have $|\prod\limits_{1\leq l\leq r_1}Q(\xi_l)|\gtrsim x^{r_1r_2}$, so that we deduce \eqref{fo30}.\par
Finally, let us prove \eqref{fo4}. This follows at once from \eqref{inversa2} taking into account that, if $r_1+1\leq j\leq r_1+r_2$,
\[
\prod\limits_{1\leq l\leq r_1+r_2\atop l\not=j}|\xi_j(x)-\xi_l(x)|\gtrsim x^{r_1}\prod\limits_{r_1+1\leq l\leq r_1+r_2\atop l\not=j}|\xi_j(x)-\xi_l(x)|,
\]
 which is a consequence of \eqref{fo2}.
\end{proof}

\begin{proposition}\label{pro19}
Let $p(x,\xi)$ be as in \eqref{1.5}. Let $\xi_1(x),\ldots,\xi_{r_1+r_2}(x)$, $r_1+r_2=m$, be its roots, defined for $x$ large, and assume \eqref{fo1}, \eqref{fo2}, as well as
\begin{equation}\label{fo2bis}
|\xi_j(x)-\xi_k(x)|\gtrsim \max\{|\xi_j(x)|,|\xi_k(x)|,x^{-1+\epsilon}\},\quad {\rm for}\ j,k=r_1+1,\ldots, r_1+r_2,
\end{equation}
for some $\epsilon>0$, as $x\to+\infty$.
Let $P$ be the operator with Weyl symbol $p(x,\xi)$. We can write, for $x$ large,
\begin{equation}\label{med}
P=\sum_{k=0}^{r_1} \eta_k(x) D^{r_1-k} \prod_{j=r_1+1}^{r_1+r_2} (D-\eta_j(x))+R,
\end{equation}
where $\eta_k\in S^{j}(\R)$, $k=1,\ldots, r_1$, $b_0(x)=1$, $\eta_j\in S^{1}(\R)$, $j=r_1+1,\ldots r_1+r_2$ (see \eqref{do4}),
\begin{equation}\label{med18}
\eta_j(x)=\sigma_j( \xi_1(x),\ldots, \xi_{r_1}(x))+O(x^{j-2})\quad{\rm for}\  j=1,\ldots, r_1
\end{equation}
and
\begin{equation}\label{med19}
\eta_j(x)=\xi_j(x)+O(x^{-1})\quad{\rm for}\  j=r_1+1,\ldots, r_1+r_2
\end{equation}
as $x\to+\infty$, whereas $R$ is a differential operator of order $r_1+r_2$ whose coefficients are in $S^{m}(\R)$ and rapidly decreasing as $x\to+\infty$, together with their derivatives.
\end{proposition}
\begin{proof}
The standard symbol of  $P$ is given in Proposition
\ref{proweyl}. On the other hand, the symbol of the first
operator in the right-hand side of \eqref{med} is computed
in Proposition \ref{pro1}. By equating the coefficients of
the terms of the same order we deduce that, modulo rapidly
decreasing functions as $x\to+\infty$, it must be
\begin{equation}\label{system}
\sum_{k=0}^{r_1+r_2}\sum_{l+h=k\atop l\leq r_1, h\leq r_2} a_l \sigma _h (\xi_{r_1+1},\ldots,\xi_{r_1+r_2})+R^\prime_k= \sum_{l+h=k\atop l\leq r_1, h\leq r_2}\eta_l(x)\sigma_h(\eta_{r_1+1},\ldots \eta_{r_1+r_2})+R^{\prime\prime}_k
\end{equation}
for $k=1,\ldots, r_1+r_2$, where the $\eta_j$'s are unknown, $a_l=a_l(x)=\sigma_l(\xi_1(x),\ldots,\xi_{r_1}(x))$,
 \[
 R^\prime_k\in {\rm span}_{\C}\{ (a_l \sigma _h (\xi_{r_1+1},\ldots,\xi_{r_1+r_2}))^{(\nu)},\ 1\leq \nu\leq k,\ l+h=k-\nu\},
\]
and
\begin{multline*}
R^{\prime\prime}_k\in{\rm span}_{\C}\{a_l \eta_{j_1}^{(m_1)}\cdots\eta_{j_h}^{(m_h)},\ l+h+m_1+\ldots +m_h=k,\ 1\leq h\leq k-1,\\ m_1+\ldots +m_h>0,\ r_1+1\leq j_1<\ldots<j_h\leq r_1+r_2\}.
\end{multline*}
We set
\begin{equation}\label{sost}
\begin{cases} \eta_j(x)= a_j(x)+\zeta_j(x)&{\rm for}\ j=1,\ldots, r_1\\ \eta_j(x)= \xi_j(x)-\zeta_j(x)&{\rm for}\ j=r_1+1,\ldots, r_1+r_2,
\end{cases}
\end{equation}
in the right-hand side of \eqref{system}. By isolating the terms which are of order $0$ or $1$
with respect to $\zeta=(\zeta_1,\ldots, \zeta_{r_1+r_2})$ we can write
\[
\sum_{l+h=k\atop l\leq r_1, h\leq r_2}\eta_l\sigma_h(\eta_{r_1+1},\ldots \eta_{r_1+r_2})=\sum_{l+h=k\atop l\leq r_1, h\leq r_2} a_l \sigma _h (\xi_{r_1+1},\ldots,\xi_{r_1+r_2})+(A\zeta)_k+X_k,
\]
where $A$ is exactly the matrix described in Lemma
\ref{lemma1}. By multi-linearity we can also write
$R^{\prime\prime}_k=Y_k+Z_k$, where $Z_k$ is the term
constant with respect to the $\zeta_j$'s. The system
\eqref{system} hence  becomes
\begin{equation}\label{system2}
\zeta=-A^{-1} X-A^{-1}Y-A^{-1}Z-A^{-1} R^\prime.
\end{equation}
Let us prove that
\begin{equation}\label{med1}
(A^{-1} Z)_j=\begin{cases}
O(x^{j-2})&{\rm if}\ j=1,\ldots, r_1\\
O(x^{-1})&{\rm if }\ j=r_1+1,\ldots,r_1+r_2.
\end{cases}
\end{equation}
as $x\to+\infty$.
We have
\begin{multline}\label{med3}
Z_k\in{\rm span}_{\C}\{a_l \xi_{j_1}^{(m_1)}\cdots\xi_{j_h}^{(m_h)},\ l+h+m_1+\ldots +m_h=k,\ 1\leq h\leq k-1,\\ m_1+\ldots +m_h>0,\ r_1+1\leq j_1<\ldots<j_h\leq r_1+r_2\}.
\end{multline}
Hence, if $j\leq r_1$ and $A^{-1}=(B_{jk})$, we have $(A^{-1}Z)_j=\sum_{k=1}^{r_1+r_2}B_{jk} Z_k$, and by \eqref{med3}, \eqref{fo1} and \eqref{fo30}
\[
B_{jk} Z_k=O(x^{j-k}x^{l+h-m_1-\ldots-m_h})=O(x^{j-2(k-l-h)})=O(x^{j-2}),
\]
because  $k-l-h\geq1$.\\
 If $j\geq r_1+1$, since $\xi^{(m_\nu)}_{j_\nu}=O(\xi_{j_\nu}(x)x^{-m_\nu})$ and $m_1+\ldots +m_h=k-k-l$, by \eqref{med3} it suffices to prove that
\begin{equation}\label{med4}
\frac{a_l(x)x^{-k+h+l}|\xi_j(x)|^{r_1+r_2-k}\xi_{j_1}(x)\cdots\xi_{j_h}(x)}{x^{r_1}\prod\limits_{r_1+1\leq l\leq r_1+r_2\atop l\not=j}|\xi_j(x)-\xi_l(x)|
}=O(x^{-1}).
\end{equation}
By \eqref{fo2bis} we have $\xi_{j_\nu}(x)= O(|\xi_{j_\nu}-\xi_j|)$ if $j_\nu\not= j$ and $\xi_j(x)=O(|\xi_j(x)-\xi_l(x)|)$ and $x|\xi_j(x)-\xi_l(x)|\to+\infty$ if $j\not=l$, $l\geq r_1+1$. Hence we obtain
\begin{equation}\label{med5}
x^{-r_1+k-h-1}\cdot\frac{|\xi_j(x)|^{r_1+r_2-k}\xi_{j_1}(x)\cdots\xi_{j_h}(x)}{\prod\limits_{r_1+1\leq l\leq r_1+r_2\atop l\not=j}|\xi_j(x)-\xi_l(x)|
}=O(1)
\end{equation}
(this is easily verified by considering separately the cases when the number of factors in the numerator is less or greater than that of the denominator). On the other hand, by \eqref{fo1} we have $|a_l(x)|x^{-k+k+l}x^{-k+h+1}=O(x^{-2(k-l-h)+1})=O(x^{-1})$, which combined with \eqref{med5} gives \eqref{med4}. This proves \eqref{med1}.

Similarly one proves that \begin{equation}\label{med2}
(A^{-1} R^\prime)_j=\begin{cases}
O(x^{j-2})&{\rm if}\ j=1,\ldots, r_1\\
O(x^{-1})&{\rm if }\ j=r_1+1,\ldots,r_1+r_2
\end{cases}
\end{equation}
as $x\to+\infty$.\par
Now, \eqref{med1} and \eqref{med2} suggest to look for functions $\zeta_j$ with asymptotic expansion
\begin{equation}\label{med6}
\zeta_j(x)\sim \sum_{\mu=0}^{+\infty} c_{j,\mu} x^{j-2-\frac{\mu}{p}},\quad{\rm for}\ j=1,\ldots,r_1,
\end{equation}
\begin{equation}\label{med7}
\zeta_j(x)\sim \sum_{\mu=0}^{+\infty} c_{j,\mu} x^{-1-\frac{\mu}{p}},\quad{\rm for}\ j=r_1+1,\ldots,r_1+r_2,
\end{equation}
where $p$ is the integer for which \eqref{puiseaux} hold. In order for the system \eqref{system2} to be solvable by an iterative argument (in the sense of formal power series), it suffices to prove that the coefficient in $(A^{-1} X)_j$, respectively $(A^{-1} Y)_j$, of $x^{j-2-\frac{k}{p}}$, respectively $x^{-1-\frac{k}{p}}$, depends only on the $c_{j,\mu}$ with $\mu<k$. To this end, consider first the term $(A^{-1}X)_j$. By the very definition of $X_k$, it belongs to the complex span of
\[
\delta_l(x)\gamma_{j_1}(x)\cdots \gamma_{j_h}(x),
\]
with  $l+h=k,\ l\leq r_1,\ h\leq r_2$, where $\delta_l=a_l$ or $\delta_l=\zeta_l$, and $\gamma_{j_\nu}=\xi_{j_\nu}$ or $ \gamma_{j_\nu}=\zeta_{j_\nu}$, and the above product contains {\it at least  two factors
of type $\zeta$}. \par
Let $1\leq j\leq r_1$. We have $(A^{-1}X)_j=\sum_{k=1}^{r_1+r_2} B_{jk} X_k$. By \eqref{fo30}, \eqref{fo1}, \eqref{med6} and \eqref{med7} we get $B_{jk} X_k=O(x^{j-k} x^{l+h-4})=O(x^{j-4})$. More generally, the same argument shows that  for a fixed $\mu\geq0$ the coefficients $c_{i,\mu}$, $i=1,\ldots, r_1+r_2$, may appear in the asymptotic expansion of $(A^{-1}X)_j$ only in the terms of degree $\leq j-2-\frac{\mu}{p}-2$.\par
Let now $r_1+1\leq j \leq r_1+r_2$.  Let us prove $B_{jk} X_k=o(x^{-1})$.
It suffices to prove that
\begin{equation}\label{med4bis}
\frac{\delta_l(x)|\xi_j(x)|^{r_1+r_2-k}\gamma_{j_1}(x)\cdots\gamma_{j_h}(x)}{x^{r_1}\prod\limits_{r_1+1\leq l\leq r_1+r_2\atop l\not=j}|\xi_j(x)-\xi_l(x)|
}=o(x^{-1}),
\end{equation}
if $l+h=k$, $l\leq r_1$, $h\leq r_2$. We can suppose that, say, $\gamma_{j_h}=\zeta_{j_h}$. Assume furthermore that there exists $1\leq \nu\leq h-1$ such that $\gamma_{j_\nu}=\zeta_{j_\nu}$. Then  we have $\delta_l(x) x^{-k+h} \gamma_{j_h}(x)=O(x^{-1})$ and
\[
\frac{|\xi_j(x)|^{r_1+r_2-k}\gamma_{j_1}(x)\cdots\gamma_{j_{h-1}}(x)x^{-r_1+k-h}}{\prod\limits_{r_1+1\leq l\leq r_1+r_2\atop l\not=j}|\xi_j(x)-\xi_l(x)|
}=o(1),
\]
where we used the same arguments as in the proof of \eqref{med5} and the fact that $\zeta_{j_\nu}(x)/|\xi_j(x)-\xi_{j_\nu}(x)|=O(x^{-1}/|\xi_j(x)-\xi_{j_\nu}(x)|)=o(1)$. This gives \eqref{med4bis}.
If instead there is no such a $\nu$, it must be $\delta_l(x)=\zeta_l(x)$. Hence $\delta_l(x) x^{-k+h+1}=O(x^{-1})$ and
\[
\frac{|\xi_j(x)|^{r_1+r_2-k}\gamma_{j_1}(x)\cdots\gamma_{j_{h}}(x)x^{-r_1+k-h-1}}{\prod\limits_{r_1+1\leq l\leq r_1+r_2\atop l\not=j}|\xi_j(x)-\xi_l(x)|
}=o(1),
\]
which still gives \eqref{med4bis}. This proves that, if $r_1+1\leq j\leq r_1+r_2$, $(A^{-1} X)_j=o(x^{-1})$ as $x\to+\infty$.  More generally, the same argument shows that  for a fixed $\mu\geq0$ the coefficients $c_{i,\mu}$, $i=1,\ldots, r_1+r_2$, may appear in the asymptotic expansion of $(A^{-1} X)_j$ only in the terms of degree $<-1-\frac{\mu}{p}$. Similary one sees that the same is true for $(A^{-1} Y)_j$.\par
Hence, the system \eqref{system2} has a formal power series solutions in the form \eqref{med6}, \eqref{med7}, as $x\to+\infty$. Now, by a classical Borel-type argument, see e.g. \cite[Proposition 18.1.3]{hormanderIII}, one can construct functions $\zeta_j\in S^{j-2}(\R)$ if $1\leq j\leq r_1$, $\zeta_j\in S^{-1}(\R)$ if $r_1+1\leq j\leq r_1+r_2$, with the asymptotic expansions in \eqref{med6} and \eqref{med7} respectively.\par The functions $\eta_j(x)$ defined in \eqref{sost} for $x$ large and extended smoothly to zero for $x<0$ then fulfill  \eqref{med} for a convenient operator $R$ having the desired properties. This concludes the proof.
\end{proof}
\section{Proof of the main result (Theorem \ref{mainteo})}
\subsection{Sufficient Condition} Let $u$ be in $\cS^\prime(\R)$ with $Pu\in\cS(\R)$. Let us prove that, under the assumptions \eqref{separazione} and \eqref{condition} we have $u\in\cS(\R)$, i.e.  $WF(u)=\emptyset$ (Proposition \ref{pro16bis}).\par We use the factorization \eqref{fact} for the Weyl symbol $p(x,\xi)$, valid for large $|x|$, and define the complex constant $\lambda_j$, $j=1,\ldots,m$ by \eqref{fact1}. Let now $(x_0,\xi_0)\not=(0,0)$. If $(x_0,\xi_0)$ does not belong to any of the rays $\{t(1,\lambda_j),\, t\in\R\setminus\{0\}\}$, with $\lambda_j\in\R$, we see from \eqref{fact}, \eqref{fact1} that $p$ is non-characteristic at $(x_0,\xi_0)$, i.e. it satisfies the estimate \eqref{med11} for some $\epsilon>0$. Hence $(x_0,\xi_0)\not\in WF(u)$ by Proposition \ref{pro16}. \par
Let now $(x_0,\xi_0)$ lie on a ray $\{t(1,\lambda_j),\ t\in\R\setminus\{0\}\}$, for some $\lambda_j\in\R$. We can suppose that $(x_0,\xi_0)=(1,\lambda_j)$ or $(x_0,\xi_0)=(-1,-\lambda_j)$. We can further reduce to the case when $(x_0,\xi_0)=(1,0)$. In fact, suppose that $(x_0,\xi_0)=(1,\lambda_j)$, and consider the linear symplectic transformation $\chi(x,\xi)=(x,\xi+\lambda_j x)$. Let $U$ be the operator associated to $\chi$ via \eqref{med10}. By \eqref{med10bis}, in order to get $(x_0,\xi_0)\not\in WF(u)$ it suffices to prove that $(1,0)\not\in WF(U^{-1} u)$. Now, we have $(p\circ\chi)^w(x,D) U^{-1} u=U^{-1} Pu\in\cS(\R)$, so that our original problem is equivalent to a similar one with $(x_0,\xi_0)$ replaced by $(1,0)$ and the symbol $p(x,\xi)$ replaced by $(p\circ\chi)(x,\xi)$. The same holds for $(x_0,\xi_0)=(-1,-\lambda)$ if we perform the preliminar symplectic transformation $\chi(x,\xi)=(-x,-\xi)$.\par
After these transformations we get a symbol, which we will continue to call $p(x,\xi)$, which has a factorization as in \eqref{fact} for $x$ large, where the new roots $\xi_j(x)$ satisfy \eqref{fact1} and \eqref{puiseaux} for other values of $\lambda_j$, as well as \eqref{separazione} and \eqref{condition}. \par Suppose now $\lambda_{j}\not=0$ for $j=1,\ldots, r_1$, and $\lambda_j=0$ for $j=r_1+1,\ldots, r_1+r_2$, with $r_1+r_2=m$.  Accordingly, by \eqref{fact1}, \eqref{separazione} we have
\begin{equation}\label{med15}
|\xi_j(x)|\asymp x\quad{\rm for}\ 1\leq j\leq r_1,\qquad |\xi_j(x)|=o(x)\quad{\rm for}\ r_1+1\leq j\leq r_1+r_2,
\end{equation}
\begin{equation}\label{med16}
|\xi_j(x)-\xi_k(x)|\gtrsim \max\{|\xi_j(x)|,|\xi_k(x)|,x^{-1+\epsilon}\}\qquad{\rm for}\ r_1+1\leq j\leq r_1+r_2,
\end{equation}
as $x\to+\infty$. Let us verify that $(1,0)\not\in WF(u)$ if $Pu\in\cS(\R)$.\par We can apply Proposition \ref{pro19} and use the factorization in \eqref{med}. We claim that the symbol of the first factor, namely
$\sum_{k=0}^{r_1} \eta_k(x) \xi^{r_1-k}
$, is non-characteristic at $(1,0)$. In fact  by \eqref{med15} we have
\[
|\sum_{k=0}^{r_1} \sigma_j(\xi_1(x),\ldots,\xi_{r_1}(x)) \xi^{r_1-k}|=\prod_{j=1}^{r_1} |\xi-\xi_j(x)|\gtrsim x^{r_1},\quad{\rm for}\ (x,\xi)\in V_{(1,0),\epsilon}
\]
if $\epsilon$ is small enough. Now by \eqref{med18} we deduce that
\[
|\sum_{k=0}^{r_1} \eta_k(x) \xi^{r_1-k}|\gtrsim x^{r_1},\quad{\rm for}\ (x,\xi)\in V_{(1,0),\epsilon}
\]
which proves the claim. Hence, since $Pu\in\cS(\R)$ and $(1,0)\not\in WF(Ru)$, by Proposition \ref{pro16} and \eqref{med} we deduce that $(1,0)$ does not belong to the wave-front set of  $\prod_{j=r_1+1}^{r_1+r_2} (D-\eta_j(x))u$. Hence, to finish the proof it is sufficient to prove that,  for every $j=r_1+1,\ldots,r_1+r_2$, if $(1,0)\not\in WF((D-\eta_j(x))u)$ then $(1,0)\not\in WF(u)$. \par
Fix such a $j$. We have from \eqref{puiseaux} that
\begin{equation}\label{puiseaux2}
\xi_{j}(x)=\sum_{-
\infty<k\leq p-1} c_{j,k}x^{k/p}\qquad {\rm for}\ x\ {\rm large}.
\end{equation}
By \eqref{med19} we have an asymptotic expansion
\begin{equation}\label{do0}
\eta_{j}(x)=\sum_{-p<k\leq p-1} c_{j,k}x^{k/p}+O(x^{-1})\qquad {\rm for}\ x\ {\rm large}.
\end{equation}
By \eqref{condition}, there exists $k>-p$ such that $c_{j,k}\not\in\R$. Let $\nu\leq p-1$ be the greatest index for which this holds. Let $Q(x)$ be a smooth real-valued function, $Q(x)=0$ for $x<0$, and $Q_j(x)=\int _0^x \sum_{\nu<k\leq p-1} c_{j,k}t^{k/p}\, dt$ for $x$ large. We can write
\begin{equation}\label{do1}
D-\eta_j(x)= e^{i Q_j(x)} (D-\tilde{\eta}_{j}(x))\circ e^{-i Q_j(x)},
\end{equation}
where $\tilde{\eta}_{j}\in S^{1}(\R) $, and
\begin{equation}\label{do}
\tilde{\eta}_{j}(x)=\eta_{j}(x)-\sum_{\nu<k\leq p-1} c_{j,k}x^{k/p}=c_{j,\nu}x^{\nu/p}+o(x^{\nu/p})
\end{equation}
 for large $x$. \par
Now, we can regard $e^{\pm i Q_j(x)}$ as symbols (independent of $\xi$) in the class $ \Gamma^0_{1,\delta}(\R)$ with $\delta=1-1/p$ (see \eqref{pro17bis}). Hence, if $(1,0)\not\in WF((D-\eta_j(x))u)$, by Proposition \ref{pro18} we get $(1,0)\not\in WF((D-\tilde{\eta}_{j}(x))u)$. On the other hand,  the symbol $\xi-\tilde{\eta}_j(x)$ belongs to $\tilde{\Gamma}^1_\delta(\R)$ (see \eqref{med12}), and it is hypoelliptic at $(1,0)$ (see \eqref{med13},\eqref{med13g}). In fact, using \eqref{do} and the fact that $c_{j,\nu}\not\in\R$ we have
\[
|\partial^\alpha_\xi\partial^\beta_x (\xi-\tilde{\eta}_{j}(x))|\leq
C_{\alpha\beta}|\xi-\tilde{\eta}_{j}(x)|\langle \xi\rangle^{-\alpha}\langle x\rangle^{-\beta+\delta\alpha},\quad\forall\alpha,\beta\in\mathbb{N},\quad x>1,\ \xi\in\R,
\]
with $\delta=\max\{0,-\nu/p\}\leq 1-1/p<1$, and $|\xi-\tilde{\eta}_{j}(x)|\gtrsim |\xi|+x^{\nu/p}$ for  $x$ large and $\xi\in\R$. Hence, by Proposition \ref{pro17} we obtain $(1,0)\not\in WF(u)$, which concludes the proof.
 \subsection{Necessary Condition} Let us assume \eqref{separazione} but suppose that \eqref{condition} fails for some $j$. Let us prove that then  there exists $u\in\cS^\prime(\R)$, $u\not\in\cS(\R)$ such that $Pu\in\cS(\R)$. By the applying the same arguments as in the proof of the sufficient condition we can assume that \eqref{condition} fails, say, for $x\to+\infty$ and $j=r_1+r_2$, with $\lambda_{r_1+r_2}=0$. Hence \eqref{puiseaux2} and therefore \eqref{do0} hold with $j=r_1+r_2$ and $c_{r_1+r_2,k}\in\R$ for every $-p<k\leq p-1$.  Let us set
 \[
 Q(x)= \int_0^x \eta_{r_1+r_2}(t)\,dt,
 \]
 and let $u\in C^\infty(\R)$ such that $u(x)=0$ for $x<0$, $u(x)= e^{i Q(x)}$ for $x>1$. Then, by \eqref{do0} with $j=r_1+r_2$
  we have $x^{-c}\lesssim |u(x)|\lesssim x^{c}$ when $x>1$, for some $c>1$. Hence $u\in\cS^\prime(\R)$, $u\not\in\cS(\R)$. Moreover, $Pu=0$ for $x<0$, whereas for $x$ large we have $(D-\eta_{r_1+r_2}(x))u=0$ and therefore, from \eqref{med}, $Pu=Ru$. Now, we also have  $|\partial^\alpha u(x)|\lesssim x^{c+|\alpha|}$ when $x>1$, because $\eta_{r_1+r_2}\in S^1(\R)$, so that $Ru$ is rapidly decreasing as $x\to+\infty$ together with its derivatives. This implies that $Pu\in\cS(\R)$ and concludes the proof.
\section{Examples and remarks}
We give in the following some examples and remarks.
\begin{remark}\label{rem5.1}\rm
The assumption $c_{m,0}\not=0$ for the Weyl symbol of $P$
can be eliminated, provided we submit preliminarily $P$ to
the conjugation by a metaplectic operator $U$. Namely, let
$U$ be associated to the symplectic map
$\chi(x,\xi)=(x-\lambda \xi,\xi)$, $\lambda\in\R$. In view
of \eqref{med10}, the Weyl symbol of $\tilde{P}=U^{-1}PU$
is given by
\[
(p\circ \chi)(x,\xi)=\sum_{\alpha+\beta\leq m} c_{\alpha,\beta}(x+\lambda \xi)^\beta \xi^\alpha
\]
with new coefficient
\[
\tilde{c}_{m,0}=\sum_{\alpha+\beta= m} c_{\alpha,\beta}\lambda^\beta\not=0
\]
for a generic choice of $\lambda$. On the other hand, the global regularity of $P$ and $\tilde{P}$ are equivalent,
 because $U,U^{-1}:\cS(\R)\to\cS(\R)$, $U,U^{-1}:\cS^\prime(\R)\to\cS^\prime(\R)$. We may correspondingly
 re-formulate condition \eqref{separazione} and Theorem \ref{mainteo}.
\end{remark}
\begin{example}\label{exa5.2}\rm  ({\bf Globally elliptic operators}, cf. Grushin \cite{C4}, Shubin \cite{shubin}, Helffer \cite{Helffer:1}). Assume all the roots $\lambda_j$, $j=1,\ldots,m$, in \eqref{1.5} satisfy ${\rm Im}\, \lambda_j\not=0$. This is equivalent to the so-called global ellipticity of the symbol:
\begin{equation}\label{5.1}
|p(x,\xi)|\geq \epsilon (1+|x|+|\xi|)^m\quad{\rm for}\ |x|+|\xi|\geq R,
\end{equation}
for some $\epsilon>0$, $R>0$. In fact, if we write $p(x,\xi)$ as a sum of homogeneous terms
\begin{equation}\label{5.2}
p(x,\xi)=\sum_{0\leq j\leq m} p_j(x,\xi),\quad p_j(x,\xi)=\sum_{\alpha+\beta=j} c_{\alpha,\beta} x^\beta\xi^\alpha,
\end{equation}
the condition \eqref{5.1} is equivalent to
\begin{equation}\label{5.3}
p_m(x,\xi)\not=0\quad{\rm for}\ (x,\xi)\not=(0,0).
\end{equation}
The corresponding operators are globally regular, since \eqref{condition} is obviously satisfied. Basic example is the harmonic oscillator of Quantum Mechanics:
\[
P= D^2+x^2.
\]
\end{example}
\begin{example}\label{exa5.3}\rm
Consider now the case when $\lambda_j\in\R$ for some $j$. Assume for the moment that all the real roots $\lambda_j$ are distinct, that is $\partial_\xi p_m(x,\lambda_j x)\not=0$ for $x\not=0$. We may apply Theorem \ref{mainteo}. By factorization we have
\[
\xi_j(x)=\lambda_j x +c_j+O(x^{-1})
\]
with $c_j=p_{m-1}(x,\lambda_j x)/\partial_\xi p_m(x,\lambda_j x)$. Then $P$ is globally regular if and only if ${\rm Im}\, c_j\not=0$ for every real $\lambda_j$. Consider for example the elementary polynomial
\[
p(x,\xi)=\xi-x+c.
\]
Limiting attention to the corresponding homogeneous equation, a classical solution of $Du-xu+c u=0$ is given by $u(x)={\rm exp}[ ix^2/2-icx]$. If ${\rm Im}\, c\not=0$, then $u\not\in\cS^\prime(\R)$, whereas if $c\in\R$, then $u\in \cS^\prime(\R)$, $u\not\in\cS(\R)$, contradicting the global regularity.
\end{example}
\begin{example}\label{exa5.4}\rm ({\bf Quasi-elliptic operators}, cf. Grushin \cite{C4}, Boggiatto, Buzano, Rodino \cite{boggiatto-buzano-rodino}, Cappiello, Gramchev, Rodino \cite{CGR}). We pass now to consider the case when two, or more, real roots coincide. For simplicity we shall assume $\lambda_j=0$ for all $j=1,\ldots,m$, that is $c_{\alpha,\beta}=0$ for all $\alpha,\beta$ with $\alpha+\beta=m$, apart from $c_{m,0}=1$. We may then consider the largest rational number $q>1$ such that $\alpha+q\beta\leq m$ for all $(\alpha,\beta)$ with $c_{\alpha,\beta}\not=0$, and write
\begin{equation}\label{5.4}
p(x,\xi)=\sum_{\alpha+q\beta\leq m} c_{\alpha,\beta} x^\beta\xi^\alpha,
\end{equation}
with $c_{\alpha,\beta}\not=0$ for some $(\alpha,\beta)$ with $\alpha+q\beta=m$, $\alpha\not=m$. For the moment, we understand $q<\infty$, that is a term with $\beta\not=0$ actually exists. Note also that $q\geq \frac{m}{m-1}$. Ths symbol in \eqref{5.4} is called (globally) quasi-elliptic if
\begin{equation}\label{5.5}
p_{m,q}(x,\xi)=\sum_{\alpha+q\beta=m} c_{\alpha,\beta} x^\beta \xi^\alpha\not=0\quad {\rm for}\ (x,\xi)\not=(0,0).
\end{equation}
This implies in particular $c_{0,{m/q}}\not=0$ and $m/q$ is
a positive integer. The corresponding operators are
globally regular, as proved for example in \cite{C4},
\cite{boggiatto-buzano-rodino}, \cite{CGR}. Computing the
Puiseaux expansion \eqref{puiseaux}, we have
\begin{equation}\label{5.6}
\xi_j(x)= r_j^{\pm}|x|^{1/q}+o(|x|^{1/q})\quad {\rm for}\ x\to\pm\infty.
\end{equation}
where $r_j^{\pm}$ are the roots in $\C$ of $p_{m,q}(\pm1,r)=0$. We deduce that \eqref{5.5} is satisfied if and only if ${\rm Im}\, r_j^{\pm}\not=0$ for all the roots. \par
We then recapture the global regularity from Theorem \ref{mainteo}, provided condition \eqref{separazione} is satisfied, i.e. the roots $r_j^+$, or equivalently $r^-_j$, are distinct, cf. \eqref{1.x}. As example, consider the Airy-type operator
\[
P=D^2+cx,
\]
which is globally regular if and only if ${\rm Im}\, c\not=0$.
\end{example}
\begin{example}\rm \label{exa5.5}
Let the symbol $p(x,\xi)$ be of the form \eqref{5.4}, with distinct roots $r^{\pm}_j$ in \eqref{5.6}, but assume now $r^+_j\in\R$, or $r^-_j\in\R$, for some $j$. The condition \eqref{separazione} is satisfied, and we are led to determine the subsequent terms in the Puiseaux expansion \eqref{1.x}. We refer to Bliss \cite{bliss} for general rules of computations, and we limit ourselves here to the example
\begin{equation}\label{5.7}
p(x,\xi)=\xi^m-A\xi^r-x
\end{equation}
with $0\leq r<m$, $A\in\R$. The expansion in \eqref{5.6} reads in this case
\[
\xi_j(x)=e_j x^{1/m}+o(|x|^{1/m})
\]
where $e_j$, $j=1,\ldots,m$, are the $m$th roots of $1$. To
fix ideas, assume $m$ even; then we have for $x\to+\infty$
the two real roots $\pm 1$. It is easy to compute
\[
\xi_{\pm}(x)=\pm x^{1/m}+c_{\pm} x^s+o(x^s),\quad x\to+\infty,
\]
with $s=(r+1-m)/m$ and $c_{\pm}=\pm A/m$. Note that $1/m>s>-1$. Hence if ${\rm Im}\, A\not=0$, the condition \eqref{condition} is satisfied and Theorem \ref{mainteo} gives global regularity. Let us test this result on the corresponding homogeneous equation
\begin{equation}\label{5.8}
D^m u-A D^r u-x u=0.
\end{equation}
Every solution $u\in\cS^\prime(\R)$, or $u\in\cS(\R)$ of the equation can be regarded as inverse Fourier transform of a solution $v\in\cS^\prime(\R)$, respectively $v\in\cS(\R)$, of
\[
D v+ (x^m-Ax^r) v=0,
\]
having the (classical) solution
\[
v(x)= e^{iA x^{r+1}/(r+1)-i x^{m+1}/(m+1)}.
\]
If ${\rm Im}\, A\not=0$, then $v\not \in\cS^\prime (\R)$ or
$v\in\cS(\R)$, depending on $r$ and $A$, that agrees with
the global regularity of the operator. If $A\in\R$, then
$v\in \cS^\prime(\R)$, $v\not\in\cS(\R)$, contradicting
global regularity.
\end{example}
\begin{example}\label{exa5.6}\rm ({\bf Multi-quasi-elliptic operators}, cf. Boggiatto, Buzano, Rodino \cite{boggiatto-buzano-rodino}).
By conjugation with Fourier transform, which we may
consider as a metaplectic operator, we can treat operators
with symbol of the form \eqref{5.4} where the role of $x$
and $\xi$ is exchanged. Global regularity is granted by
\eqref{5.5} or \eqref{5.6} with ${\rm Im}\,
r^{\pm}_j\not=0$ where we exchange again $x$ with $\xi$;
relevant examples of the corresponding operators are
\[
D+ix^m,\quad D^2+x^{2m},
\]
for $m>1$. Multi-quasi-elliptic symbols are products of the
symbols of this form, the quasi-elliptic symbols in Example
\ref{exa5.4} and the globally elliptic symbols in Example
\ref{exa5.2}, possibly perturbed by terms in the interior
of the Newton polygon generated in this way, see
\cite{boggiatto-buzano-rodino} for details and equivalent
definitions. Under the condition \eqref{separazione}, we
recapture from Theorem \ref{mainteo} the result of global
regularity in \cite{boggiatto-buzano-rodino}. Limiting
again to an example, consider the symbol
\[
\xi^3+i x\xi^2+x^2.
\]
We have $\xi_1(x)=-i x+o(|x|)$,
 $\xi_2(x)=\sqrt{i} x^{1/2}+o(|x|^{1/2})$,
  $\xi_3(x)=-\sqrt{i}x^{1/2}+o(|x|^{1/2})$, writing $\pm \sqrt{i}$ for
the two roots of $i$. Hence \eqref{condition} is satisfied,
and the corresponding operator is globally regular.
\end{example}
\begin{example}\rm \label{exa5.7}({\bf SG-elliptic operators},
cf. Parenti \cite{rD1}, Cordes \cite{rD2}, Schrohe
\cite{rD6}, Schulze \cite[Section 1.4]{rD3}). The case
$q=\infty$ in Example \ref{exa5.4} corresponds to the case
of the operators with constant coefficients, that we treated
in Section 1. More generally, we can consider symbols of
the following form, with $m\geq0, n\geq0$:
\begin{equation}\label{5.9}
p(x,\xi)=\sum_{\alpha\leq m,\atop \beta\leq n}
c_{\alpha,\beta} x^\beta \xi^\alpha
\end{equation}
where we assume $c_{m,n}=1$. We say that the symbol
\eqref{5.9} is $SG-$elliptic if
\begin{equation}\label{5.10}
|p(x,\xi)|\geq \epsilon (1+|x|)^n(1+|\xi|)^m,\quad{\rm
for}\ |x|+|\xi|\geq R,
\end{equation}
for some $\epsilon>0$, $R\geq0$. In fact \eqref{5.10} is
equivalent to the following couple of conditions:
\begin{equation}\label{5.11}
\sum_{\alpha\leq m} c_{\alpha,n} \xi^\alpha\not=0\quad{\rm
for}\ \xi\not=0,
\end{equation}
\begin{equation}\label{5.12}
\sum_{\beta\leq n} c_{m,\beta} x^\beta\not=0\quad{\rm for}\
x\not=0.
\end{equation}
We know from \cite{rD1,rD2,rD6,rD3} that $SG$-elliptic
operators are globally regular. Willing to apply our
Theorem \ref{mainteo}, we first observe that factorizing
$p(x,\xi)$ we obtain the roots
\begin{equation}\label{5.13}
\xi_j(x)=r_j+o(1),\quad j=1,\ldots,m,
\end{equation}
where $r_j$ are exactly the roots of \eqref{5.11}. Since
\eqref{5.11} is equivalent to ${\rm Im}\, r_j\not=0$ for
every $j$, \eqref{5.11} implies \eqref{condition} for
\eqref{5.13}. Similarly we can argue on the local
ellipticity condition \eqref{5.12}, by using Remark
\ref{rem5.1}. So we recapture global regularity, in the
case when \eqref{separazione} is satisfied (i.e. the roots
of \eqref{5.11}, \eqref{5.12} are distinct).
\end{example}
\begin{example}\label{exa5.8}\rm
When, for $p(x,\xi)$ as in \eqref{5.9}, the equations
\eqref{5.11}, \eqref{5.12} admit real roots, we are led to
study terms in the Puiseaux expansion with negative
exponents. Consider for example the symbol
\begin{equation}\label{5.14}
(1+x^n)\xi^m+1,\quad m>0\ {\it and}\ n>0\ \textit{even
integers}.
\end{equation}
The local ellipticity condition \eqref{5.12} is satisfied,
whereas \eqref{5.11} reduces to $\xi^m=0$. We obtain
\[
\xi_j(x)=e_j |x|^{-n/m}+o(|x|^{-n/m}),\quad j=1,\ldots,n
\]
where $e_j$ are the $m$th roots of $-1$. Since $m$ is even,
${\rm Im}\, e_j\not=0$, and condition \eqref{condition} is
satisfied if $-n/m>-1$, i.e. $n<m$. So we may conclude
global regularity for the corresponding Weyl operator, in
this case. When $n\geq m$, both \eqref{separazione} and
\eqref{condition} fail. In fact, the operator is Fuchsian
at $\infty$ for $m=n$, and regular at $\infty$ if $n>m$.
Hence the solutions of the homogeneous equation belong to
$\cS^\prime (\R)$ and do not belong to $\cS(\R)$. We
address to Camperi \cite{rD17} for a general class of
symbols of the form \eqref{5.14}.
\end{example}
\begin{remark}\label{rem5.9}\rm
According to the definition of Schwartz \cite{schwartz}, a
differential operator with polynomial coefficients $P$ is
{\it hypoelliptic} in $\R$ if for every open subset
$\Omega\subset\R$ we have:
\begin{equation}
\label{5.15} u\in\mathcal{D}'(\Omega)\ {\rm and}\ Pu\in
C^\infty(\Omega) \Rightarrow u\in C^\infty(\Omega).
\end{equation}
We observe first that that every operator of the form
\eqref{1.1} with $a_{m,0}\not=0$ is obviously hypoelliptic.
In the general case, considered in Remark \ref{rem5.1}, we
have to take into account the coefficient $Q(x)$ of the
leading derivative, say of order $p<m$:
\[
Q(x)=\sum_{\beta\leq m-p} c_{p,\beta} x^\beta.
\]
Let us write $x_1,\ldots,x_r$, $r\leq m-p$, for the real
zeros of $Q(x)$. The hypoellipticity of $P$ is granted in
$\R\setminus\{x_1,\ldots,x_r\}$. It is then easy to prove
that {\it the global regularity in \eqref{1.3} implies
hypoellipticity in \eqref{5.15}}. \par In fact, assume
$u\in\mathcal{D}^\prime(\Omega)$ and $Pu\in
C^\infty(\Omega)$. Let $\phi\in C^\infty_0(\Omega)$,
$\phi(x)=1$ in a neighborhood of the points
$x_1,\ldots,x_r$ which belong to $\Omega$. We know that
${\rm sing-supp}\,u\subset\{x_1,\ldots,x_r\}\cap\Omega$,
hence $(1-\phi)u\in C^\infty(\Omega)$. On the other hand
$\phi u\in\mathcal{E}^\prime(\Omega)\subset\cS^\prime(\R)$
and $P(\phi u)=Pu-P(1-\phi)u\in
C^\infty_0(\Omega)\subset\cS(\R)$, hence $\phi u\in
C^\infty_0(\R)$ by the global regularity, and therefore
$u\in C^\infty(\Omega)$.\par Note that this argument cannot
be extended to operators with polynomial coefficients in
$\R^n$, $n>1$, because the manifold where the local
ellipticity fails is in general non-compact.\par Note also
that Schwartz hypoellipticity does not imply global
regularity. In dimension 1, consider for example the
operators $P_1= D_x$, $P_2= D_x+ x^h$, $h\geq2$, which are
obviously Schwartz hypoelliptic, but not globally regular.
In particular, $P_2$ keeps the Schwartz hypoellipticity
after any possible metaplectic conjugation. For a study of
the Schwartz hypoellipticity at $x_1,\ldots,x_r$, we
address to Kannai \cite{kannai}, where a necessary and
sufficient condition was given under an asymptotic
separation condition, similar to \eqref{separazione}.

\end{remark}
\section*{Acknowledgments}
We would like to thank Antonio J. Di Scala very much for several useful discussions.

\end{document}